\documentclass[12pt,reqno]{amsart}
\usepackage{amssymb,amsmath,amsthm,enumerate,bbm}

\DeclareMathOperator{\supp}{supp}

\DeclareMathOperator{\BMO}{BMO}

\newcommand{\abs}[1]{\lvert#1\rvert}
\newcommand{\Abs}[1]{\left\lvert#1\right\rvert}
\newcommand{\norm}[1]{\lVert#1\rVert}
\newcommand{\babs}[1]{\pmb{\lvert}#1\pmb{\rvert}}

\newcommand{\jap}[1]{\langle#1\rangle}

\newcommand{\bbT}{{\mathbb T}}
\newcommand{\bbR}{{\mathbb R}}
\newcommand{\bbC}{{\mathbb C}}

\newcommand{\bbN}{{\mathbb N}}
\newcommand{\bbZ}{{\mathbb Z}}

\newcommand{\bA}{\mathbf{A}}
\newcommand{\bu}{\mathbf{u}}
\newcommand{\bh}{{\mathbf{h}}}
\newcommand{\bq}{{\mathbf{q}}}
\newcommand{\bv}{{\mathbf{v}}}
\newcommand{\bw}{{\mathbf{w}}}
\newcommand{\bg}{{\mathbf{g}}}
\newcommand{\bU}{{\mathbf{U}}}

\newcommand{\calB}{\mathcal{B}}

\newcommand{\Sch}{\mathbf{S}}

\numberwithin{equation}{section}

\theoremstyle{plain}
\newtheorem{theorem}{\bf Theorem}[section]
\newtheorem*{theorem*}{Theorem 1.1$'$}
\newtheorem{lemma}[theorem]{\bf Lemma}

\theoremstyle{definition}

\theoremstyle{remark}
\newtheorem*{remark*}{\bf Remark}
\newtheorem{remark}[theorem]{\bf Remark}
\newtheorem{example}[theorem]{\bf Example}

\newcommand{\loc}{\mathrm{loc}}

\newcommand{\Hank}{\Gamma} 
\newcommand{\bHank}{{\mathbf{\Gamma}}}

\newcommand{\fh}{\widecheck{h}}
\newcommand{\fbh}{\widecheck{\mathbf h}}

\DeclareFontFamily{U}{mathx}{\hyphenchar\font45}
\DeclareFontShape{U}{mathx}{m}{n}{<5> <6> <7> <8> <9> <10>
<10.95> <12> <14.4> <17.28> <20.74> <24.88> mathx10}{}
\DeclareSymbolFont{mathx}{U}{mathx}{m}{n}
\DeclareFontSubstitution{U}{mathx}{m}{n}
\DeclareMathAccent{\widecheck}{0}{mathx}{"71}

\begin{document}

\title[Sharp estimates for  Hankel operators]{Sharp estimates for singular values of Hankel operators}

\author{Alexander Pushnitski}
\address{Department of Mathematics, King's College London, Strand, London, WC2R~2LS, U.K.}
\email{alexander.pushnitski@kcl.ac.uk}

\author{Dmitri Yafaev}
\address{Department of Mathematics, University of Rennes-1,
Campus Beaulieu, 35042, Rennes, France}
\email{yafaev@univ-rennes1.fr}



\begin{abstract} 
We consider compact Hankel operators realized in $ \ell^2(\bbZ_+)$
as infinite matrices $\Gamma$ with matrix elements $h(j+k)$. 
Roughly speaking, we show that, for all $\alpha>0$, the singular values $s_{n}$ of $\Gamma$ satisfy the bound  $s_{n}= O(n^{-\alpha})$ as $n\to \infty$ provided $h(j)= O(j^{-1}(\log j )^{-\alpha})$
as $j\to \infty$.  These estimates on $s_{n}$ are sharp in the power scale of  $\alpha$.
Similar results are obtained for Hankel operators $\bHank$ realized in $L^2(\bbR_+)$ as integral operators with kernels $\bh(t+s)$. 
In this case the estimates of singular values of $\bHank$ are determined 
by the behavior  of $\bh(t)$ as $t\to 0$ and as $t\to \infty$.
\end{abstract}

\date{1 December 2014}

\maketitle

\section{Introduction}\label{sec.a1}

\subsection{Basic notions}

The theory of Hankel operators exists in two representations: discrete and continuous. 
In the discrete representation, one starts with a   sequence of complex numbers 
$\{h(j)\}_{j=0}^\infty$, and  one formally defines the Hankel operator $\Hank(h)$ in 
$\ell^2(\bbZ_+)$ as the 
``infinite matrix'' $\{h(j+k)\}_{j,k=0}^\infty$, i.e.
\begin{equation}
(\Gamma(h) u) (j)= \sum_{k=0}^\infty h(j+k) u (k), \quad u = (u (0), u (1), \ldots).
\label{eq:a5}
\end{equation}
The Nehari-Fefferman  theorem says that the Hankel operator $\Hank(h)$ is bounded on $\ell^2(\bbZ_+)$ if and only if the  \emph{symbol} of $\Hank(h)$, defined by 
\begin{equation}
\omega (\mu)=\sum_{j=0}^\infty h(j)\mu^j, \quad \abs{\mu}=1,
\label{eq:sy}
\end{equation}
belongs to the class 
$\BMO(\bbT)$ of functions of the bounded mean oscillation on the unit circle $\bbT$.
A simple sufficient condition for the boundedness of $\Hank(h)$ is the estimate
$h(j) =O(j^{-1})$
 as $j\to \infty$.

In the continuous  representation, one starts with a function 
 $\bh\in L^1_{\rm loc} (\bbR_{+})$ ($\bbR_+=(0,\infty)$),  and the 
\emph{integral Hankel operator} $\bHank(\bh)$ in $L^2(\bbR_+)$
  with the  \emph{kernel} $\bh$ is given by the formula
\begin{equation}
(\bHank(\bh)\bu)(t)=\int_0^\infty \bh(t+s)\bu(s)ds .
\label{eq:HH}
\end{equation}
Similarly to the discrete case,  the Hankel operator $\bHank(\bh)$ is bounded on $L^2(\bbR_+)$  if and only if the corresponding symbol belongs to the class 
$\BMO(\bbR)$.
A simple sufficient condition for the boundedness of $\bHank(\bh)$ is the estimate
$$
\abs{\bh(t)}\leq C/t, \quad t>0.
$$
Throughout the paper, we will use the boldface font for objects 
associated with the continuous representation.

\subsection{A conjecture}

Let $\omega$ be the symbol \eqref{eq:sy} of a Hankel operator $\Hank(h)$, and let $\Sch_p $ be the Schatten class
of compact operators (see Section~1.4). 
V.~Peller has shown that
\begin{equation}
\Hank(h)\in \Sch_p 
\quad \Leftrightarrow \quad 
\omega\in B^{1/p}_{pp}(\bbT),
\quad p>0,
\label{z1a}
\end{equation}
where 
 $B^{1/p}_{pp}(\bbT)$ is the Besov space; 
see the book \cite{Peller} for the proof, the history and references to the relevant papers of other authors. 
By using the real interpolation between Besov spaces, V.~Peller has deduced 
from \eqref{z1a} a \emph{necessary and sufficient} condition (given by the finiteness of the expression \eqref{ca9}) 
 for the estimate
\begin{equation}
s_n(\Hank(h))=O(n^{-\alpha}), \quad n\to\infty, \quad \alpha>0,
\label{z1b}
\end{equation}
for the singular values of $\Hank(h)$; we refer again to the book \cite{Peller} for the details. 
This condition is stated in terms of the inclusion of $ \omega$ into a   certain 
function class of the Besov-Lorentz type denoted  in \cite[Section 6.5]{Peller} by  $\mathfrak{B}_{p,\infty}^{1/p}$ where $p=1/\alpha$.
Similar results exist in the continuous case.

Our aim here is to give a simple \emph{sufficient} condition for \eqref{z1b} directly 
in terms of the sequence $h(j)$. 
It is expected that the faster rate of convergence  $h(j)\to0$ as $j\to\infty$ 
implies the faster rate of convergence of the singular values 
$s_n(\Hank(h))\to0$ as $n\to\infty$.  We show that the correct condition on the decay of $h(j)$ is given  in the logarithmic scale. To be more precise, we discuss the following

\noindent
\textbf{Conjecture:}
\begin{equation}
h(j)=O(j^{-1}(\log j)^{-\alpha})
\quad
\Rightarrow
\quad
s_n(\Hank(h))=O(n^{-\alpha}),\quad \alpha>0.
\label{a1g}
\end{equation}
Let us consider two special cases that motivate this conjecture. 

(i) $\alpha=0$.
It is well known (see, e.g.,  \cite{Peller}) 
that the Hankel operator $\Gamma(h)$ (the Hilbert matrix) corresponding to the sequence
\begin{equation}
h(j)=\frac1{j+1}, \quad j\geq0,
\label{a1d}
\end{equation}
  is bounded (but not compact).
It follows that
\begin{equation}
h(j)=O(1/j), \quad j\to\infty
\quad \Rightarrow \quad 
\Hank(h)\in\calB
\label{a1h}
\end{equation}
($\calB$ is the class of bounded operators).

(ii) $\alpha>1/2$. 
A Hankel operator  $\Gamma$ belongs to the  Hilbert-Schmidt class $\Sch_2$ if and only if
\begin{equation}
\sum_{n=1}^\infty s_n(\Hank(h))^2=\sum_{j=0}^\infty (j+1)\abs{h(j)}^2< \infty.
\label{z12}
\end{equation}
Obviously, the series in the r.h.s.  
converges if $h(j)= O(j^{-1}(\log j)^{-\alpha})$ for some $\alpha> 1/2$, 
and the series in the l.h.s. converges if 
$s_n(\Hank(h))=O(n^{-\alpha})$ for some $\alpha>1/2$.

The main purpose of this paper is to show that  \emph{the above conjecture is partially true}.
More precisely, we prove that 
\emph{the conjecture is true for $\alpha< 1/2$;}
for $\alpha\geq 1/2$, we prove that the conclusion of \eqref{a1g} becomes true if we 
assume that the sequence $h(j)$ behaves sufficiently regularly, i.e. if we 
impose appropriate additional assumptions on the sequence of differences 
$h(j+1)-h(j)$ and on its higher order iterates. 
We also obtain analogous results in the continuous case. 
Precise statements are given in Section~\ref{sec.a}. 

Let us comment on the proofs. 
For $\alpha\geq 1/2$ we deduce our results from  Peller's necessary and sufficient condition 
$ \omega\in \mathfrak{B}_{p,\infty}^{1/p}$ for the estimate \eqref{z1b}. 
For $\alpha<1/2$ our approach is more direct and relies on the real interpolation 
between the cases (i) and (ii) (where $\alpha$ is arbitrarily close to $1/2$) mentioned above.

\subsection{Discussion}

Our results are quite simple and efficient. 
However, our sufficient condition for \eqref{z1b} is  far from being necessary because we do not take into account
possible oscillations of $h(j)$. 
In order to illustrate this point, let us observe   that in the limit $\alpha\to0$ 
our results reduce to the well-known implication \eqref{a1h}. 
There are many sequences that fail to satisfy  $h(j)=O(j^{-1})$   but such that $\Hank(h)\in\mathcal B$. Consider, for example,  $h(j)=n^{-2}$ for $j=n^4$, $n\in\bbN$, and $h(j)=0$ otherwise. Obviously,  $h(j)=O(j^{-\alpha})$ for $\alpha\leq 1/2$ only. However the function \eqref{eq:sy} is bounded in the unit disc and hence $\Hank(h)\in\mathcal B$ by the Nehari theorem.

To a large extent, our aim is to  provide  technical tools for \cite{II}, 
where we study the   asymptotic behavior of eigenvalues of compact self-adjoint Hankel operators. 
In particular, in \cite{II} we show
that for the sequence
\begin{equation}
h(j)=j^{-1}(\log j)^{-\alpha} , \quad j\geq2, \quad \alpha>0,
\label{z13}
\end{equation}
  the  asymptotics 
\begin{equation}
s_n(\Hank(h))=
v(\alpha)n^{-\alpha}
+o(n^{-\alpha}), \quad n\to\infty,
\label{z14}
\end{equation}
holds with    the explicit constant $v(\alpha)$  given by
\begin{equation}
v(\alpha) = 2^{-\alpha} \pi^{1-2\alpha}\Big(B(\tfrac{1}{2\alpha},\tfrac12)\Big)^\alpha,
\label{eq:V}
\end{equation}
where $B(\cdot,\cdot)$ is the standard Beta function. 
Clearly, \eqref{z13}, \eqref{z14} show that the 
exponent $\alpha$ in the right-hand side of \eqref{a1g} is optimal in the class of Hankel operators we consider.

\subsection{Schatten classes}

Let us recall some basic information on ideals of compact operators in a Hilbert space (see the books \cite{BSbook,GK}). 
We denote by 
$\calB$ the set of all bounded operators,   $\norm{\cdot}$ is the operator norm; $\Sch_\infty$ is the set of all compact operators. 
Let $\{s_n(\Hank)\}_{n=1}^\infty$ be  the non-increasing sequence of singular values of $\Hank \in \Sch_\infty$
(i.e. the eigenvalues of $\sqrt{\Hank^*\Hank}$).
For $p>0$, the Schatten class $\Sch_p$ and the weak  Schatten class $\Sch_{p,\infty}$ of compact operators  
are defined by the conditions 
$$
\Hank\in\Sch_p
\quad \Leftrightarrow \quad 
\norm{\Hank}_{\Sch_p}^p:=
\sum_{n=1}^\infty s_n(\Hank)^p<\infty
$$
and
$$
\Hank\in\Sch_{p,\infty}
\quad \Leftrightarrow \quad 
\norm{\Hank}_{\Sch_{p,\infty}}:=
\sup_{n\geq1} n^{1/p}s_n(\Hank)
<\infty.
$$
The classes $\Sch_p  $ and  $\Sch_{p,\infty}$ are the ideals of the algebra $\calB$ 
with the  quasi-norms
$\norm{\cdot}_{\Sch_{p }}$ and $\norm{\cdot}_{\Sch_{p,\infty}}$. 
The class $\Sch_{p,\infty}^0$  
is the closed linear subspace of $\Sch_{p,\infty}$ defined by 
$$
\Hank\in \Sch_{p,\infty}^0
\quad\Leftrightarrow\quad
\lim_{n\to\infty} n^{1/p}s_n(\Hank)=0.
$$
Equivalently, $\Sch_{p,\infty}^0$ may be defined as the closure of the set of all 
finite rank operators in the quasi-norm 
$\norm{\cdot}_{\Sch_{p,\infty}}$. We have
$$
\Sch_p \subset \Sch_{p,\infty}^0\subset \Sch_{p,\infty} \subset  \Sch_\infty.
$$

\subsection{Plan of the paper}

We state  our main results in Section~2. Their proofs are given in Sections~3 and 4 for the continuous and discrete cases, respectively. It is convenient to start the proofs with the continuous case because integration by parts 
is more visual than the corresponding procedure (the Abel transformation for series)  in the discrete case. 

Throughout the rest of the paper, $C$ (possibly with indices) denotes 
constants in estimates, and the value of $C$ may change from line to line. 
Notation $\abs{X}$ means the Lebesgue measure of the set $X\subset\bbT$ or of $X\subset\bbR$.
We make a standing assumption that the exponents $p>0$ and $\alpha>0$ are 
related by $\alpha=1/p$.

\section{Main results}\label{sec.a}

\subsection{Discrete representation}
Let the Hankel operator $\Hank(h)$ be defined by formula \eqref{eq:a5} in the space $\ell^2(\bbZ_+)$. We first justify the conjecture  \eqref{a1g}
for $\alpha< 1/2$. This case turns out to be significantly simpler. 
Here $p>2$   and  $\Sch_{p,\infty}\not\subset\Sch_{2}$.

\begin{theorem}\label{thm.z1}
Let $\alpha< 1/2$ and let $\{h(j)\}_{j=0}^\infty$ be a sequence of complex numbers
such that 
\begin{equation}
h(j)=O(j^{-1}(\log j)^{-\alpha}), \quad
j\to\infty.
\label{z3}
\end{equation}
Then   the singular values of the corresponding Hankel operator $\Hank(h)$ satisfy the estimate
\begin{equation}
s_n(\Hank(h))=O(n^{-\alpha}), \quad n\to\infty.
\label{ca6}
\end{equation}
Moreover, there is a constant $ C(\alpha)$ such that
$$
\norm{\Hank(h)}_{\Sch_{p,\infty}}
\leq
C(\alpha)
\sup_{j\geq0}
(j+1)(\log (j+2))^{\alpha}\abs{h(j)}, \quad p=1/\alpha.
$$
\end{theorem}

Next, consider the case $\alpha\geq 1/2$. 
Here, besides \eqref{z3}, we require some additional assumptions. 
For a sequence $h$, we denote by $h^{(m)}$, $m=0,1,2,\dots$  the sequences
of iterated differences. 
Those are the sequences defined iteratively by setting $h^{(0)} (j)=h  (j)$  
and 
$$
h^{(m)}(j)=h^{(m-1)}(j+1)-h^{(m-1)}(j), \quad j\geq0.
$$
The number of times we need to iterate will be determined
by the integer 
\begin{equation}
M(\alpha)=
\begin{cases}
[\alpha]+1,& \text{ if } \alpha\geq 1/2,
\\
0, & \text{ if } \alpha < 1/2,
\end{cases}
\label{c5}
\end{equation}
where $[\alpha]=\max\{ m\in\bbZ_{+}: m\leq \alpha\}$.
The following result includes Theorem~\ref{thm.z1} as a particular case.

\begin{theorem}\label{thm.a1}
Let $\alpha>0$, and let $M=M(\alpha)$ be defined by \eqref{c5}.
Let  $h$ be a sequence of complex numbers that satisfies
\begin{equation}
h^{(m)}(j)
=
O(j^{-1-m}(\log j)^{-\alpha}), \quad j\to\infty,
\label{ca5}
\end{equation}
for all $m=0,1,\dots,M$.
Then the estimate
\begin{equation}
s_n(\Hank(h))=O(n^{-\alpha}), \quad n\to\infty,
\label{eq:ca5}
\end{equation}
holds,
and there is a constant $C(\alpha)$ such that 
\begin{equation}
\norm{\Hank(h)}_{\Sch_{p,\infty}}
\leq
C(\alpha)
\sum_{m=0}^M
\sup_{j\geq0}
(j+1)^{1+m}(\log (j+2))^{\alpha}\abs{h^{(m)}(j)}, \quad p=1/\alpha.
\label{z7}
\end{equation}
\end{theorem}

\begin{theorem}\label{cr.a3}
If \eqref{ca5} holds with $o$ instead of $O$ for all $m=0,1,\dots,M$, then we have
\begin{equation}
s_n(\Hank(h))=o(n^{-\alpha}), \quad n\to\infty.
\label{ca6o}
\end{equation} 
\end{theorem}

Theorems~\ref{thm.z1}, \ref{thm.a1} and \ref{cr.a3} are proven in Section~\ref{sec.ca}.
As was already mentioned, Theorem~\ref{thm.z1} admits a
  direct   proof based on the real interpolation between  the cases  
  ${\Hank} (h)\in\Sch_2$ and   ${\Hank} (h)\in \calB$. 
In the proof of 
Theorem~\ref{thm.a1}, we proceed from  the results of \cite{Peller}
which give necessary 
and sufficient conditions for  $\Gamma  (h)\in\Sch_{p }$ and hence  for  $\Gamma(h)\in\Sch_{p,\infty}$
in terms of the symbol \eqref{eq:sy} of this operator. We prove that 
under the hypothesis of Theorem~\ref{thm.a1}, 
such conditions
are satisfied. Theorem~\ref{cr.a3} is deduced from Theorem~\ref{thm.a1} by simple approximation arguments.

\begin{remark}
\begin{enumerate}
\item
As already mentioned above,
relations  \eqref{z13} and \eqref{z14} show that
the exponent $\alpha$ in \eqref{ca6} is optimal.

\item
Theorem~\ref{thm.a1} is false if no  conditions on the iterated differences 
$h^{(m)} (j)$ are imposed. 
Further, while our condition  on the exponent $M(\alpha)$ is probably not optimal, it is not far from being so.  
Indeed, Example~\ref{exa} shows that, for $\alpha \geq  2$,  one cannot
take $M(\alpha)=[\alpha]-2$ in Theorem~\ref{thm.a1}.

\item
 Some sufficient conditions for the inclusion $\Hank(h)\in \Sch_1$,
stated in terms of the sequences $h$, $h^{(1)}$ and $h^{(2)}$ were found in \cite{Bonsall}. 
They are similar in spirit to Theorem~\ref{thm.a1}.

\item
If
$h(j)= O (j^{-\gamma})$ for some $\gamma>1$ as $j\to\infty$ and if some conditions on the iterated differences $h^{(m)} (j)$ are satisfied, then one can expect that the singular values $s_{n} (\Gamma(h))$ decay faster than any power of $n^{-1}$ as $n\to\infty$. 
In fact,  H.~Widom  showed  in \cite{Widom} that for $h(j)=  (j+1)^{-\gamma}$, $\gamma>1$,  the 
corresponding Hankel operator
 $\Gamma(h)$  is non-negative   and its eigenvalues converge to zero \emph{exponentially} fast:
$$
\lambda_n^+(\Gamma(h))=\exp(-\pi \sqrt{2\gamma n}+o(\sqrt{n})), \quad n\to\infty.
$$
Some additional results in this direction were obtained in \cite{Parf}.
\end{enumerate}
\end{remark}

If a sequence $h(j)$ satisfies  \eqref{ca5} for $m=0$ and if $\zeta\in \bbT$, then the sequence  
$ \zeta^j h(j)$ satisfies the same condition; but for $m>0$ this  implication is no longer  true. Nevertheless we have the following simple generalization of Theorems~\ref{thm.a1} and   \ref{cr.a3}.

\begin{theorem}\label{cr.a4}
Let   the sequences  $h_1,h_2,\ldots,h_L$   
satisfy the hypothesis of Theorem~\ref{thm.a1} (resp. Theorem~\ref{cr.a3}), and 
let $\zeta_\ell \in\bbT$, $\ell=1,\dots,L$. 
Then the estimate  \eqref{ca6} (resp. \eqref{ca6o}) holds true 
for the Hankel operator $\Gamma(h)$ corresponding to the sequence 
\begin{equation}
h(j)=\sum_{\ell=1}^L \zeta_\ell^j h_\ell(j), \quad \zeta_{\ell}\in\bbT.
\label{z15}
\end{equation}
\end{theorem}

\begin{proof}
For a sequence $h(j)$  and for $\zeta \in\bbT$, we denote by 
$q_{\zeta}$ the sequence $q_{\zeta}(j) = \zeta^j h (j)$. 
Let $U_\zeta$ be the unitary operator in $\ell^2 (\bbZ_{+})$ given by 
$$
(U_\zeta f)(j)=\zeta^j f(j), \quad j\geq0.
$$
By inspection we have
\begin{equation}
\Hank(q_{\zeta})=U_\zeta \Hank(h) U_\zeta
\label{z10d}
\end{equation}
and therefore 
$s_n(\Hank(q_{\zeta}))=s_n(\Hank(h))$
for all $n$. 

Since the classes $\Sch_{p,\infty}^0$ and   $\Sch_{p,\infty} $ are linear spaces,   estimates  \eqref{ca6} and  \eqref{ca6o} for
the operators $\Hank(h_\ell)$ extend to the sum
$$
\Hank(h)=\sum_{\ell=1}^L U_{\zeta_\ell}\Hank(h_\ell)U_{\zeta_\ell}.
$$
This concludes the proof.
\end{proof}

Of course, instead of a finite sum in \eqref{z15} one can 
consider infinite series or integrals.

\subsection{Continuous representation}\label{sec.a1a}

Now  the Hankel operator $\bHank(\bh)$ is defined by formula \eqref{eq:HH} in the space $L^2(\bbR_+)$. 

In the discrete representation, the spectral properties of $\Hank(h)$  are determined by the asymptotic
behaviour of the sequence $h(j)$ as $j\to\infty$. 
In the continuous representation, the behaviour of the kernel $\bh(t)$ for $t\to0$ 
and for $t\to\infty$ as well as the local singularities of $\bh$ 
contribute to the spectral properties of the Hankel operator $\bHank(\bh)$.
Therefore we impose some local smoothness conditions on $\bh(t)$, but
 our main   attention will be directed towards the behaviour of $\bh(t)$ as $t\to0$ and $t\to\infty$. 

Recall (see, e.g.,  \cite{Peller}) that the Carleman operator, corresponding to the kernel
$
\bh(t)=1/t,
$
is bounded.
From here, similarly to \eqref{a1h}, one easily obtains
\begin{equation}
\abs{\bh(t)}\leq C/t 
\quad \Rightarrow \quad
\bHank(\bh) \in \calB.
\label{a5b}
\end{equation}
In the continuous case, the 
Hilbert-Schmidt condition is given by  
\begin{equation}
\sum_{n=1}^\infty s_n(\bHank(\bh))^2=\int_0^\infty t\abs{\bh(t)}^2dt<\infty.
\label{z11}
\end{equation}
Of course, this condition is satisfied if  $\bh \in L^2_\loc(\bbR_+)$ and $\bh(t)=O(t^{-1} |\log t|^{-\alpha} )$ for some $\alpha> 1/2$ as $t\to0$ and for $t\to\infty$.

This suggests that one should consider kernels $\bh(t)$ that are logarithmically
``smaller'' than $1/t$ both for $t\to0$ and for $t\to\infty$. 
Indeed, 
the analogue of the estimate \eqref{z3} in the continuous case is 
\begin{equation}
\abs{\bh (t)}\leq A_0 t^{-1}\jap{\log t}^{-\alpha}, \quad t>0;
\label{cb5C}
\end{equation}
here and in what follows we use the notation $\jap{x}=(\abs{x}^2+1)^{1/2}$. We start with the ``continuous analogue" of  Theorem~\ref{thm.z1}. 

\begin{theorem}\label{thm.z1C}
Let     $\bh$ be a complex valued function in $L^\infty_\loc(\bbR_+)$ satisfying the 
estimate  \eqref{cb5C} with some $\alpha< 1/2$.
Then for  the singular values of the corresponding Hankel operator $\bHank(\bh)$ 
one has
\begin{equation}
s_n(\bHank(\bh))=O(n^{-\alpha}), \quad n\to\infty.
\label{z9}
\end{equation}
Moreover, $\norm{\bHank(\bh)}_{\Sch_{p,\infty}}\leq C(\alpha) A_{0}$ for some constant $ C(\alpha)$  and $ p=1/\alpha$.
\end{theorem}

For $\alpha\geq 1/2$, we also need additional conditions on the derivatives
$\bh^{(m)}(t)=(d/dt)^m \bh(t)$. 
The following result is the ``continuous analogue" of  Theorem~\ref{thm.a1}. 
It includes Theorem~\ref{thm.z1C} as a particular case.

\begin{theorem}\label{thm.a2}
Let $\alpha>0$ and  let $M=M(\alpha)$ be the integer given by \eqref{c5}.
Let $\bh$ be a complex valued function in $L^\infty_\loc(\bbR_+)$; 
if $\alpha\geq 1/2$, suppose also 
that $\bh\in C^M(\bbR_+)$.
Assume that  
\begin{equation}
\abs{\bh^{(m)}(t)}\leq A_m t^{-1-m}\jap{\log t}^{-\alpha}, \quad t>0,
\label{cb5}
\end{equation}
with some constants $A_0,\dots,A_M$
for all $m=0,\dots,M$.
Then the singular values of the corresponding Hankel operator $\bHank(\bh)$ satisfy
\eqref{z9}
and, for some constant $ C(\alpha)$,
$$
\norm{\bHank(\bh)}_{\Sch_{p,\infty}}
\leq
C(\alpha)
(A_0+\dots+A_M), \quad p=1/\alpha.
$$
\end{theorem}

\begin{theorem}\label{cr.a7}
In addition to the hypothesis of Theorem~$\ref{thm.a2}$, assume   that 
\begin{equation}
\abs{\bh^{(m)}(t)}=o( t^{-1-m}\jap{\log t}^{-\alpha}) \quad \text{ as $t\to0$ and as $t\to\infty$.}
\label{z10}
\end{equation}
Then 
\begin{equation}
s_n(\bHank(\bh))=o(n^{-\alpha}), \quad n\to\infty.
\label{z9C}
\end{equation}
\end{theorem}

Theorems~\ref{thm.z1C}, \ref{thm.a2} and \ref{cr.a7} are proven in Section~\ref{sec.cb}. 
Their proofs are similar to those in the discrete case. 
In particular, Theorem~\ref{thm.z1C} admits a
  direct   proof based on the real interpolation between  the 
Hilbert-Schmidt condition \eqref{z11}  and the  sufficient condition \eqref{a5b} for the boundedness of $\bHank(\bh)$. 
In the proof of 
Theorem~\ref{thm.a2}, we proceed from  the results of \cite{Peller}
which give necessary 
and sufficient conditions for  ${\bHank} ({\bf h})\in\Sch_{p }$ and hence  for  ${\bHank} ({\bf h})\in\Sch_{p,\infty}$. We prove that 
under the hypothesis of Theorem~\ref{thm.a2}, 
such conditions
are satisfied. Theorem~\ref{cr.a7} is deduced from Theorem~\ref{thm.a2} by simple approximation arguments.

\begin{remark}
\begin{enumerate}
\item
The exponent $\alpha$ in \eqref{z9} is optimal.  
Indeed,
let $\bh(t)$ be a sufficiently smooth real valued function such that for some $\alpha>0$
\begin{equation}
\bh(t)
=
t^{-1}| \log t |^{-\alpha} 
\label{z10a}
\end{equation}
for all sufficiently small $t$,  and $\bh(t)=0$  for all sufficiently large $t$. 
Then it follows from the results of \cite{II} that 
\begin{equation}
s_n(\bHank(\bh))=v(\alpha)   n^{-\alpha}+o(n^{-\alpha}), 
\quad n\to\infty,
\label{z10b}
\end{equation}
where the constant $v(\alpha)$ is given by \eqref{eq:V}.
Similarly, if \eqref{z10a} holds for all large $t$ and $\bh(t)=0$  for all small $t$, 
then again by the results of \cite{II} we obtain \eqref{z10b}. 

 \item
Some sufficient conditions for the estimate $s_n(\bHank(\bh))=O(n^{-\alpha})$ 
in terms of the smoothness  of $\bh$ were obtained in \cite{GLP}, see, e.g.,  Corollary~4.6 there.
These conditions require that $\bh(t)$ vanish very fast as $t\to\infty$ but allow for 
some singular behaviour as $t\to0$. 
These results are somewhat similar to Theorem~\ref{thm.a2}
but are less sharp. 
\end{enumerate}
\end{remark}

For a function $\bh(t)$ and for $a\in \bbR$, let us denote $\bq_a (t)= e^{i a t} \bh (t)$. 
If $\bh$ satisfies  \eqref{cb5} for some $m>0$, then $\bq_a$ does not necessarily 
satisfy the same condition. 
Nevertheless, similarly to the discrete case,  the following simple argument allows 
us to extend our results to $\bHank(\bh_a)$. 
Let the unitary  operator ${\bf U}_a$ in $L^2 ({\bbR}_{+})$ be defined   by the formula
$({\bf U}_a   \mathbf{f})(t)= e^{iat}  \mathbf{f}(t)$.  
The role of \eqref{z10d} is now played by  the identity 
$$
\bHank(\bq_a)=\bU_a \bHank(\bh) \bU_a.
$$
It follows that  the singular values of the operators $\bHank(\bq_a)$ and $\bHank(\bh)$ coincide. 
Reasoning as in the proof of Theorem~\ref{cr.a4}, we obtain the following generalization of Theorems~\ref{thm.a2}
and \ref{cr.a7}.

\begin{theorem}\label{cr.a4C}
Let the functions  $\bh_1, \bh_2,\ldots,\bh_L$  satisfy the hypothesis of Theorem~$\ref{thm.a2}$ 
(resp. of Theorem~$\ref{cr.a7}$), and 
let $a_\ell \in\bbR$, $\ell=1,\dots,L$. 
Then  for the Hankel operator $ {\bHank}({\bf h} )$ with the 
kernel
\begin{equation}
\bh(t)=\sum_{\ell=1}^L e^{ia_\ell t}\bh_\ell(t)
\label{z15C}
\end{equation}
the estimate  \eqref{z9} (resp.  \eqref{z9C}) holds true.
\end{theorem}

Of course, instead of a finite sum in \eqref{z15C} one can 
consider infinite series or integrals.

Note that the results in the discrete and continuous cases are not quite independent. 
In principle, each one of them can be obtained from another   one  
through the Laguerre transform (see, e.g., the book \cite{Peller}) or by linking  
the symbols of the operators $\Gamma(h)$ and $ {\bHank}({\bf h})$ through a conformal map
from the unit disc onto the upper half-plane. 
However,  technically it is simpler    to carry out derivations in each case independently.

\section{Continuous representation}\label{sec.cb}

Recall that  the Hankel operator $\bHank(\bh)$ is defined by formula \eqref{eq:HH} in the space $L^2(\bbR_+)$. 
Here we prove Theorems~\ref{thm.z1C}, \ref{thm.a2} and \ref{cr.a7}.

\subsection{The case $\alpha<1/2$}\label{sec.cb1}
We will use weighted $L^p$ classes on $\bbR_+$ with the weight $\bv(t)=1/t$:
$$
\bg\in L^p_\bv(\bbR_+) 
\quad\Leftrightarrow\quad 
\norm{\bg}_{L^p_\bv}^p=\int_{0}^\infty \abs{\bg(t)}^p \bv(t)dt<\infty,
\quad
\bv(t)=1/t,
$$
and the corresponding weak class
\begin{equation}
\bg\in L^{p,\infty}_\bv(\bbR_+) 
\quad\Leftrightarrow\quad
\norm{\bg}_{L^{p,\infty}_\bv}^p
=
\sup_{s>0} s^p \int_{t: \abs{\bg(t)}>s} \bv(t)dt<\infty.
\label{eq:Lor}
\end{equation}
By definition, for $p=\infty$ the weighted class $L_\bv^\infty$ coincides
with the usual (unweighted) $L^\infty$ class. 

Below we will use the real interpolation method (the ``$K$-method''),
see, e.g. \cite[Section~3.1]{BL} for the details. 
A pair  $(X_0,X_1)$  of quasi-Banach spaces is called compatible,   if both $X_0$ and $X_1$
are continuously embedded into the same Hausdorff topological vector space. 
Real interpolation with the parameters $0<\theta<1$ and $1\leq q\leq \infty$
between a compatible pair of quasi-Banach spaces $(X_0,X_1)$ 
yields an intermediate quasi-Banach space $(X_0,X_1)_{\theta,q}$. In particular, we have
\begin{equation}
(L^2_\bv,L^\infty_\bv)_{\theta,\infty}=L^{p,\infty}_\bv,
\quad
(\Sch_2,\calB)_{\theta,\infty}=\Sch_{p,\infty}, 
\quad 
\theta=1- 2/{p}.
\label{eq:int}
\end{equation}
If $(X_0,X_1)$ and $(Y_0,Y_1)$ are two compatible pairs of quasi-Banach spaces
and if $T$ is a bounded linear map from $X_0$ to $Y_0$ and from $X_1$ to $Y_1$, then 
the real interpolation method ensures the boundedness of $T$ as a map 
from $(X_0,X_1)_{\theta,q}$ to $(Y_0,Y_1)_{\theta,q}$.

\begin{lemma}\label{lma.cb2}
Let $\bv(t)=1/t$, and let $\bh:\bbR_+\to \bbC$ be a measurable function such that $\bh/\bv\in L^{p,\infty}_\bv$ 
  with some $p>2$. 
Then $\bHank(\bh)\in \Sch_{p,\infty}$ and 
\begin{equation}
\norm{\bHank(\bh)}_{\Sch_{p,\infty}}\leq C_p\norm{\bh/\bv}_{L^{p,\infty}_\bv}.
\label{cb3a}
\end{equation}
\end{lemma}

\begin{proof}
The case $\bh=\bv$ corresponds to the Carleman operator, which has the norm $\pi$. 
From here we obtain that if $\bh/\bv\in L^\infty$, then $\bHank(\bh)\in \calB$, and 
$$
\norm{\bHank(\bh)}
\leq 
\pi \norm{\bh/\bv}_{L^\infty}
=
\pi \sup_{t>0} t\abs{\bh(t)}.
$$
On the other hand, we have the Hilbert-Schmidt relation \eqref{z11}. 
Thus, the linear map
\begin{equation}
\bh/\bv \mapsto  \bHank(\bh)
\label{ca4a}
\end{equation}
is bounded from $L^\infty_\bv=L^\infty$ to $\calB$ 
and from $L^2_\bv$ to $\Sch_2$. 
In view of \eqref{eq:int}, we see
 that the map \eqref{ca4a} is bounded from $L^{p,\infty}_\bv$ to $\Sch_{p,\infty}$, 
and the estimate \eqref{cb3a} holds true. 
\end{proof}

\begin{proof}[Proof of Theorem~$\ref{thm.z1C}$]
Since 
$\abs{\bh(t)/\bv(t)}\leq A_0\jap{\log t}^{-\alpha}$
and 
$$
\int_{A_0\jap{\log t}^{-\alpha}>s} \bv(t)dt
=
\int_{ \jap{\log t} <  (A_0 /s)^p} t^{-1}dt
\leq 
C A_0^p s^{-p},\quad s>0,
$$
it follows from definition \eqref{eq:Lor} that $\bh/\bv\in  L^{p,\infty}_\bv$. 
So it remains to use Lemma~\ref{lma.cb2}.
\end{proof}
As a by-product of the above argument, we also obtain

\begin{theorem}\label{thm.cb3}
For all $p\geq2$, one has
\begin{equation}
\norm{\bHank(\bh)}_{\Sch_p}^p
\leq
C_p
\int_0^\infty t^{p-1}\abs{\bh(t)}^p dt.
\label{eq:ci}
\end{equation}
\end{theorem}

\begin{proof}
Let us choose the interpolation parameter $ q= {p}$ and use that
$$
(L^2_\bv,L^\infty_\bv)_{\theta,p}=L^{p,p}_\bv=L^{p}_\bv,
\quad
(\Sch_2,\calB)_{\theta,p}=\Sch_{p,p}=\Sch_{p}, 
\quad 
\theta=1- 2/{p}.
$$
Then considering again the mapping \eqref{ca4a},   we see that
$$
\norm{\bHank(\bh)}_{\Sch_p}^p
\leq
C_p\norm{\bh/\bv}_{L^p_\bv}^p
=
C_p
\int_0^\infty t^{p-1}\abs{\bh(t)}^p dt,
$$
as required.
\end{proof}

Theorem~\ref{thm.cb3} can also be
proven by the complex interpolation method which shows that \eqref{eq:ci} holds
with $C_p =\pi^{p-2}$.

\subsection{The case $\alpha\geq 1/2$}\label{sec.cb2}

Let $\bw\in C_0^\infty(\bbR_+)$ be a function with the properties $\bw\geq0$, 
$\supp \bw=[1/2,2]$ and 
\begin{equation}
\sum_{n\in\bbZ} \bw(t/2^n)=1, \quad \forall t>0.
\label{cb7b}
\end{equation}
For $n\in\bbZ$, let $\bw_n(t)=\bw(t/2^n)$.
For a function $\bh\in L^1_\loc(\bbR_+)$ and for $n\in\bbZ$, set 
\begin{equation}
\fbh_n(x):=\int_0^\infty \bh(t)\bw_n(t)e^{ixt}dt, \quad x\in\bbR. 
\label{cb7a}
\end{equation}

\begin{theorem}\label{Pe1C}\cite[Theorem~6.7.4]{Peller}
Let $\bh\in L^1_{\rm loc} (\bbR_{+})$.
The estimate 
\begin{equation}
\norm{\bHank(\bh)}_{\Sch_p}^p
\leq
C_p
\sum_{n\in\bbZ} 2^n \int_{-\infty}^\infty\abs{\fbh_n(x)}^p dx
\label{cb8}
\end{equation}
holds,
so that $\bHank(\bh)\in \Sch_p$ if  the r.h.s.  in \eqref{cb8} is finite.
\end{theorem}

The convergence of the series in \eqref{cb8} means that the symbol of the operator $\bHank(\bh)$ 
belongs to the Besov class $B^{1/p}_{pp} (\bbR)$.
Further, we have

\begin{theorem}\label{Pe2C} 
Let $\bh\in L^1_{\rm loc} (\bbR_{+})$. Suppose that
\begin{equation}
\babs{\bh}_p^p
:=
\sup_{s>0}s^p\sum_{n\in\bbZ} 2^n
\abs{\{x\in\bbR: \abs{\fbh_n(x)}>s\}} <\infty.
\label{cb9}
\end{equation}
Then $\bHank(\bh) \in \Sch_{p,\infty}$ and
$$
\norm{\bHank(\bh)}_{\Sch_{p,\infty}}
\leq
C_p
\babs{\bh}_p. 
$$
 \end{theorem}
 
 In the discrete case (see  
 Theorem~\ref{Pe2D} below), 
 this theorem  is proven in \cite[Theorem 6.4.4]{Peller}. 
In  the continuous case,  the proof   is exactly the same, up to   trivial changes in notation. For a given $p$ one chooses some $p_{0}$ and $p_1$ such that $p_{0}< p < p_1$ and uses estimates \eqref{cb8} with $p=p_0$ and $p=p_1$. Then one applies the real interpolation method to these estimates choosing
 the interpolation parameters $\theta$, $q$  such that $ 1/p= (1-\theta)/p_0+ {\theta}/{p_0}$
 and $q=\infty$. 
 
 The results of \cite{Peller} also show that if $\bHank(\bh)\in \Sch_p$ (resp.  if $\bHank(\bh) \in \Sch_{p,\infty}$), 
 then the r.h.s. of    \eqref{cb8} (resp. of \eqref{cb9}) is  necessary finite, although we will not need these facts.

  Our goal is to check that under the assumptions  of Theorem~$\ref{thm.a2}$ the expression  \eqref{cb9} is finite.

\begin{lemma}\label{lma.cb1}
Assume the hypothesis of Theorem~$\ref{thm.a2}$.  
Then for any $q>1/M$  and  for all $n\in\bbZ$ the functions \eqref{cb7a} satisfy the estimates
\begin{align}
\norm{\fbh_n}_{L^\infty}
&\leq 
\int_{2^{n-1}}^{2^{n+1}} \abs{\bh(t)}dt,
\label{cb12}
\\
2^n \norm{\fbh_n}_{L^q}^q
&\leq 
C_q\biggl(\sum_{m=0}^M \int_{2^{n-1}}^{2^{n+1}} t^m \abs{\bh^{(m)}(t)}dt\biggr)^q
\label{cb13}
\end{align}
with a constant $C_q$ independent of $n$.
\end{lemma}

\begin{proof}
The first bound is a direct consequence of the definition \eqref{cb7a}
of $\fbh_n$ and of the properties $0\leq \bw_n\leq1$ and $\supp \bw_n=[2^{n-1},2^{n+1}]$.
In order to obtain the second bound, we write 
\begin{equation}
2^n \norm{\fbh_n}_{L^q}^q=
2^n \int_{\abs{x}\leq 2^{-n}}\abs{\fbh_n(x)}^qdx
+
2^n \int_{\abs{x}\geq 2^{-n}} \abs{\fbh_n(x)}^q dx
\label{cb14}
\end{equation}
and estimate the two terms in the r.h.s. separately. 
For the first term, we use \eqref{cb12}:
\begin{equation}
2^n \int_{\abs{x}\leq 2^{-n}}\abs{\fbh_n(x)}^qdx
\leq
2\norm{\fbh_n}^q_{L^\infty}
\leq
2
\biggl(
\int_{2^{n-1}}^{2^{n+1}} \abs{\bh(t)}dt
\biggr)^q.
\label{cb15}
\end{equation}
In order to estimate the second term in the r.h.s. of  \eqref{cb14}, we integrate 
by parts $M$ times in the definition \eqref{cb7a} of $\fbh_n$: 
\begin{multline}
\fbh_n(x)
=
(ix)^{-M}\int_0^\infty \bh(t) \bw_n(t) (d/dt)^M e^{ixt}dt
\\
=
(-ix)^{-M}\int_0^\infty(\bh(t)\bw_n(t))^{(M)}e^{ixt}dt.
\label{cb16}
\end{multline}
Since
\begin{equation}
\abs{\bw_n^{(k)}(t)}
=
2^{-nk}\abs{\bw^{(k)}(t/2^n)}
\leq
C_k 2^{-nk}, 
\quad k\geq0,
\quad n\in\bbZ,
\label{eq:dw}
\end{equation}
we get
\begin{multline}
\Abs{\int_0^\infty(\bh(t)\bw_n(t))^{(M)}e^{ixt}dt}
\leq
C_M \sum_{m=0}^M 
2^{-n(M-m)}\int_{2^{n-1}}^{2^{n+1}} \abs{\bh^{(m)}(t)}dt
\\
\leq
2^M
C_M 2^{-nM}
\sum_{m=0}^M 
\int_{2^{n-1}}^{2^{n+1}} t^m \abs{\bh^{(m)}(t)}dt.
\label{cb18}
\end{multline}
Combining \eqref{cb16} and \eqref{cb18}, we see  that
$$
  \abs{\fbh_n(x)} 
\leq
  C_M' \abs{x}^{-M} 2^{-nM }\sum_{m=0}^M \int_{2^{n-1}}^{2^{n+1}} t^m \abs{\bh^{(m)}(t)}dt 
$$
whence
\begin{multline*}
2^n \int_{\abs{x}\geq 2^{-n}} \abs{\fbh_n(x)}^q dx
\\
\leq C_M''
\biggl(
2^{n-nMq}\int_{\abs{x}\geq 2^{-n}} \abs{x}^{-Mq} dx
\biggr)
\biggl(  \sum_{m=0}^M \int_{2^{n-1}}^{2^{n+1}} t^m \abs{\bh^{(m)}(t)}dt\biggr)^q.
\end{multline*}
Since $Mq>1$, 
the first factor here equals $2/(Mq-1)$. 
Putting 
together the last estimate with \eqref{cb15} and using \eqref{cb14}, we get \eqref{cb13}. 
\end{proof}

\begin{proof}[Proof of Theorem~$\ref{thm.a2}$ for $\alpha \geq 1/2$]
Under  assumption \eqref{cb5} for all $m=0,\ldots, M$ we have
\begin{multline*}
\int_{2^{n-1}}^{2^{n+1}} t^m \abs{\bh^{(m)}(t)}dt \leq A_{m} 
\int_{2^{n-1}}^{2^{n+1}} t^{-1} \jap{\log t}^{-\alpha}dt 
\\  
=A_{m} 
\int_{ n-1}^{n+1}   \jap{x}^{-\alpha}dx \leq cA_{m}   \jap{n}^{-\alpha};
\end{multline*}
here we assume that $\log$ is the base 2 logarithm, $\log=\log_2$. 
Fix some $q\in(M^{-1}, \alpha^{-1})$; then   
it follows from \eqref{cb12}, \eqref{cb13} that
\begin{align}
\norm{\fbh_n}_{L^\infty}&\leq C A_0\jap{n}^{- \alpha}, \quad n\in \bbZ,
\label{cb19}
\\
2^n \norm{\fbh_n}_{L^q}^q&\leq C \bA^q \jap{n}^{-\alpha q}, \quad n\in \bbZ,
\label{cb20}
\end{align}
with some constant $C$ and $\bA=A_0+\dots+A_M$.  

Let us now estimate the functional 
$\babs{\bh}_p$ in \eqref{cb9}.  It follows from \eqref{cb19} that, for every $s>0$ and all $n\in\bbZ$ such that
\begin{equation}
\jap{n}> (C A_{0})^p s^{-p}=: N(s),
\label{cb21}
\end{equation}
   the inequality $\norm{\fbh_n}_{L^\infty}<s$ holds. Therefore 
 \begin{equation}
s^p\sum_{n\in\bbZ} 2^n 
\abs{\{x\in\bbR: \abs{\fbh_n(x)}>s\} } 
=
s^p\sum_{\jap{n}\leq N(s)} 2^n 
\abs{\{x\in\bbR: \abs{\fbh_n(x)}>s\} }.
\label{eq:cb20}
\end{equation}
Using  the obvious inequality
$$
s^q \abs{\{x\in\bbR: \abs{\fbh_n(x)}>s\} }
\leq
\norm{\fbh_n}_{L^q}^q 
$$
and the  bound \eqref{cb20}, we can   estimate  the  
expression \eqref{eq:cb20} by
$$
s^{p-q}
\sum_{\jap{n}\leq N (s)}2^n \norm{\fbh_n}_{L^q}^q 
\leq
s^{p-q}C \bA^q
\sum_{\jap{n}\leq N (s)} \jap{n}^{-\alpha q}
\leq
s^{p-q}C' \bA^q
 N(s)^{1-\alpha q}
$$
(we have taken into account here that $\alpha q<1$).
By virtue of \eqref{cb21} this expression is bounded by $C'' \bA^q$ with a constant $C''$ that does not depend on $s$.
Therefore it follows    from \eqref{cb20}  that $\babs{h}_p^p\leq C''\bA^p$.
In view Theorem~\ref{Pe2C}, this   yields the required result. 
\end{proof}

\begin{proof}[Proof of Theorem~$\ref{cr.a7}$]
Suppose first that $\bh(t)=0$ for all small and for all large $t>0$. 
Then according to Theorem~\ref{thm.a2} we have 
$s_{n}(\bHank(\bh ))=O(n^{-\beta})$ for all $\beta$ such that $M(\beta)\leq M(\alpha)$. 
Inspecting the formula \eqref{c5} for $M(\alpha)$, we find that we can always 
choose $\beta>\alpha$ with $M(\beta)=M(\alpha)$. Thus, we have 
$s_n(\bHank(\bh))=O(n^{-\beta})=o(n^{-\alpha})$ as $n\to\infty$.

Now let us consider the general case. 
Let $\chi_0,\chi_\infty\in C^\infty(\bbR_+)$ be such that 
\begin{equation}
\chi_0(t)=
\begin{cases}
1& \text{for $t\leq1/4$,}
\\
0& \text{for  $t\geq1/2$,}
\end{cases}
\quad
\chi_\infty(t)=
\begin{cases}
0& \text{for $t\leq2$,}
\\
1& \text{for $t\geq4$.}
\end{cases}
\label{a7b}
\end{equation}
Put
$$
\zeta_N(t)=\chi_0( t /N)\chi_\infty(N t), \quad N\in\bbN,
$$
and $\bh_N=\bh\zeta_N$. As shown by the first step of the proof, 
$\bHank(\bh_N)\in\Sch_{p,\infty}^0$. 
It remains to prove that
\begin{equation}
\norm{\bHank(\bh)-\bHank(\bh_N)}_{\Sch_{p,\infty}}\to0, 
\quad N\to\infty.
\label{cb24}
\end{equation}
According to Theorem~\ref{thm.a2}, 
we need to check that
\begin{equation}
\sup_{t>0} t^{1+m}\jap{\log t}^{1/p}
\Abs{\bigl(\bh(t)(1-\zeta_N(t)\bigr)^{(m)}}
\to 0
\quad {\rm} \quad N\to\infty,
\label{cb22}
\end{equation}
for all $m=0,\dots,M$. 
By the construction of $\zeta_N$, we have
$$
\sup_{t>0}t^m \abs{(1-\zeta_N(t)\bigr)^{(m)}}\leq C_m \quad {\rm and} \quad  (1-\zeta_N(t)\bigr)^{(m)}=0 \;  {\rm if} \; t\in (4/N, N/4)
$$
 for all $m\geq0$.
Therefore  our assumption 
\eqref{z10} on $\bh$ implies  \eqref{cb22}  and hence \eqref{cb24}.
\end{proof}

\begin{remark}
By the result of \cite[Theorem 4.9]{GLP} (see also \cite[Example 6.1]{Yafaev2}),
one can construct
a bounded kernel $\bh(t)$ with one  jump discontinuity at some $t=t_0>0$ 
(and vanishing identically for all sufficiently small 
and all sufficiently large $t>0$)
such that $\bHank(\bh)\in\Sch_{1 ,\infty}$ but
$\bHank(\bh)\notin\Sch_{1 ,\infty}^0$. Similarly, for every $\alpha\in\bbN$, $\alpha\geq2$, there exist kernels   $\bh\in C^{\alpha-2}$, $\bh\notin C^{\alpha-1}$,
such that  $\bHank(\bh)\in\Sch_{1/\alpha,\infty}$ but
$\bHank(\bh)\notin\Sch_{1/\alpha,\infty}^0$. 
This shows that, at least for $\alpha\in\bbN$, $\alpha\geq2$, the condition  $\bh\in C^M$
with  $M= \alpha-2$ is  not sufficient for the validity of estimate \eqref{z9}.
\end{remark}

\section{Discrete representation}\label{sec.ca}

Recall that  the Hankel operator $\Gamma (h)$ is defined by formula \eqref{eq:a5} in the space $\ell^2(\bbZ_+)$. 
Here we prove Theorems~\ref{thm.z1}, \ref{thm.a1} and \ref{cr.a3}.    
The calculations follow closely those of Section~\ref{sec.cb}, so we will be brief 
in places where there is a complete analogy and concentrate only on the points of difference.

\subsection{The case $\alpha <1/2$}\label{sec.ca1}

We introduce the weighted $\ell^p$ class with the weight $v(j)= (j+1)^{-1}$:
$$
g\in\ell^p_v 
\quad\Leftrightarrow\quad 
\norm{g}_{\ell^p_v}^p=\sum_{j=0}^\infty \abs{g(j)}^pv(j)<\infty,
\quad
v(j)=\frac{1}{j+1},
$$
and the corresponding weak class
$$
g\in\ell^{p,\infty}_v 
\quad\Leftrightarrow\quad 
\norm{g}_{\ell^{p,\infty}_v}^p
=
\sup_{s>0} s^p \sum_{j: \abs{g(j)}>s} v(j)<\infty.
$$
For a sequence $h$, we denote by $h/v$ the sequence $\{(j+1)h(j)\}_{j=0}^\infty$.

\begin{lemma}\label{lma.ca2}
Let $h$ be a sequence of complex numbers such that $h/v\in \ell_v^{p,\infty}$ for some $p>2$.
Then $\Hank(h)\in \Sch_{p,\infty}$ and 
$$
\norm{\Hank(h)}_{\Sch_{p,\infty}}\leq C\norm{h/v}_{\ell^{p,\infty}_v}.
$$
\end{lemma}

\begin{proof}
As in the continuous case, the result follows by real interpolation between the estimates
$$
\norm{\Hank(h)} \leq \pi \norm{h/v}_{\ell^\infty}=\pi\norm{h/v}_{\ell^\infty_v}
$$
(which corresponds to the   bound $\norm{\Hank(h)} \leq\pi$ for the Hilbert matrix \eqref{a1d}),
and the Hilbert-Schmidt relation  \eqref{z12}.
\end{proof}

\begin{proof}[Proof of Theorem~$\ref{thm.z1}$]
Since $\abs{h(j)/v(j)}\leq C(\log(j+2))^{-\alpha }$, the required statement
follows from the elementary fact that the sequence $\{(\log(j+2))^{-\alpha}\}_{j=0}^\infty$
belongs to the class $\ell^{p,\infty}_v$ for $p=1/\alpha$. 
\end{proof}

Similarly to Theorem~\ref{thm.cb3}, we also have

\begin{theorem}
For all $p\geq2$, one has
$$
\norm{\Hank(h)}_{\Sch_p}^p
\leq
\pi^{p-2}
\sum_{j=0}^\infty (j+1)^{p-1}\abs{h(j)}^p.
$$
\end{theorem}

\subsection{The case $\alpha \geq 1/2$}\label{sec.ca2}
Here we prove Theorem~\ref{thm.a1} for $0<p\leq2$.
Let $w\in C_0^\infty(\bbR_+)$ be a function with the properties $w\geq0$, $\supp w=[1/2,2]$
and 
$$
\sum_{n=0}^\infty w(t/2^n)=1, \quad \forall t\geq1.
$$
Observe that the summation is  over $n\in\bbZ_+$ here, while it is over all $n\in\bbZ$ in \eqref{cb7b}.
Denote $w_n(j)=w(j/2^n)$ for $n\geq1$ and let $w_0$ be defined by 
$w_0(0)=w_0(1)=1$, $w_0(j)=0$ for $j\geq2$.
For a sequence of complex numbers $h=\{h(j)\}_{j\geq0}$, 
denote by $\fh_n$ the polynomial
\begin{equation}
\fh_n(\mu)=\sum_{j=0}^\infty w_n(j)h(j)\mu^j, 
\quad 
\mu\in\bbT, \quad n\geq0.
\label{ca8}
\end{equation}

Let us recall two results due to V.~Peller. 
The first one follows from Theorems~6.1.1, 6.2.1 and 6.3.1 in \cite{Peller}.

\begin{theorem}\label{Pe1D} 
The estimate
\begin{equation}
\norm{\Hank(h)}_{\Sch_{p}}^p
\leq 
C_p
\sum_{n=0}^\infty 2^n 
\int_{-\pi}^{\pi}
\abs{\fh_n(e^{i\theta})}^p d\theta, \quad p>0,
\label{ca5a}
\end{equation}
holds,
so that $\Hank(h)\in \Sch_p$ if  the r.h.s.  in \eqref{ca5a} is finite.
\end{theorem}

The next result is deduced from Theorem~\ref{Pe1D} by the real  interpolation method using the retract arguments (see, e.g., the book \cite[Section~6.4]{BL}).

\begin{theorem}\label{Pe2D}\cite[Theorem~6.4.4]{Peller}
Let
\begin{equation}
\babs{h}_p^p
=
\sup_{s>0}s^p\sum_{n=0}^\infty 2^n 
\abs{\{\theta\in [-\pi, \pi): \abs{\fh_n(e^{i\theta})}>s\}}.
\label{ca9}
\end{equation}
Then $\Hank(h) \in \Sch_{p,\infty}$ and
\begin{equation}
\norm{\Hank(h)}_{\Sch_{p,\infty}}
\leq
C_p
\babs{h}_p. 
\label{ca5b}
\end{equation}
 \end{theorem}

\begin{remark}\label{Ne} 
The results of \cite{Peller} also show that if $\Hank(h)\in \Sch_p$ or $\Hank(h) \in \Sch_{p,\infty}$, then the r.h.s. of    \eqref{ca5a} or  \eqref{ca9} are  necessary finite.
 \end{remark}

  Our goal is to show that under the assumptions  of Theorem~$\ref{thm.a1}$ the expression  \eqref{ca9} is finite. Note that, for the sequence $h(j)= j^{-1} (\log j)^{-\alpha}$, $j\geq 2$, the symbol \eqref{eq:sy} is singular at the point $\mu=1$. Therefore this point requires a special treatment.
  
Let us display two elementary identities. 
The first one is   the ``summation by parts formula":
\begin{equation}
\sum_{j=0}^\infty u(j) v^{(M)} (j)= (-1)^M \sum_{j=0}^\infty u^{(M)}(j) v (j+M)
\label{eq:ibp}
\end{equation}
where it is assumed that at least one of the sequences $u$ or $v$ 
vanishes  for $j=0,\ldots, M-1$ and for all large $j$. 
The second one is the  variant of the Leibniz rule for the product $(u v)  (j) = u (j)  v  (j)$: 
\begin{equation}
(u v)^{(M)} (j)
=   
\sum_{m=0}^M \binom{M}{m}u^{(M-m)}(j+m ) v^{(m)} (j).
\label{eq:ibp1}
\end{equation}

\begin{lemma}\label{lma.ca1}
Assume the hypothesis of Theorem~$\ref{thm.a1}$. Then for any $q>1/M$
and for all $n\in\bbN$ such that $2^{n-1}\geq M$ one has the estimates
\begin{align}
\norm{\fh_n}_{L^\infty}
&\leq 
\sum_{j=2^{n-1}}^{2^{n+1}} \abs{h(j)},
\label{ca21}
\\
2^n \norm{\fh_n}_{L^q}^q
&\leq 
C_q\biggl(\sum_{m=0}^M \sum_{j=2^{n-1}-M}^{2^{n+1}} (1+j)^m \abs{h^{(m)}(j)}\biggr)^q.
\label{ca22}
\end{align}
\end{lemma}

\begin{proof}
The first estimate follows from the fact that 
$0\leq w_n(j)\leq1$ for all $j$ and $w_n(j)=0$ for $j\leq 2^{n-1}$ and for $j\geq 2^{n+1}$. 
To estimate the $L^q$ norm, we write
\begin{equation}
2^{n}(2\pi) \norm{\fh_n}_{L^q}^q
=
2^n\int_{\abs{\theta} < 2^{-n}} \abs{\fh_n(e^{i\theta})}^q d\theta
+
2^n\int_{\abs{\theta}\geq2^{-n}} \abs{\fh_n(e^{i\theta})}^q d\theta
\label{ca13}
\end{equation}
and estimate each term separately. 
For the first term, we use the estimate \eqref{ca21}:
\begin{equation}
2^n
\int_{\abs{\theta} < 2^{-n}}
\abs{\fh_n(e^{i\theta})}^q d\theta
\leq 
2
\norm{\fh_n}_{L^\infty}^q
\leq
2 \Big(\sum_{j=2^{n-1}}^{2^{n+1}} \abs{h(j)}\Big)^q.
\label{ca13a}
\end{equation}
In order to estimate the second integral in \eqref{ca13}, 
we need to perform a summation by parts calculation. 
Let us set $\mu (j)=\mu ^j$, then the iterated difference is $\mu^{(M)}(j)=(\mu-1)^M \mu(j)$. 
Using the definition \eqref{ca8} of $\fh_n$ 
and the summation by parts formula \eqref{eq:ibp} 
for sequences $u(j)=\mu (j)$, $v(j)=w_{n} (j) h(j)$,  we obtain that 
\begin{multline}
\fh_n(\mu)
=
(\mu-1)^{-M}
\sum_{j=0}^\infty 
w_n(j)h(j) \mu^{(M)}(j)
\\
=
(1-\mu)^{-M} 
\sum_{j=0}^\infty 
(w_n h)^{(M)}(j)\mu(j+M).
\label{ca13b}
\end{multline}
Since   (cf. \eqref{eq:dw})
$$
\abs{w^{(k)}_n(j)}
\leq 
C_k2^{-nk}, 
\quad 
n\geq2,
\quad
k\geq0,
$$
 it follows from the Leibniz rule \eqref{eq:ibp1} that 
  $$
\abs{(w_n h)^{(M)}(j)}
\leq
C_M\sum_{m=0}^M 2^{-n(M-m)}\abs{h^{(m)}(j)}.
$$
Substituting this into \eqref{ca13b} and using the fact that 
$w^{(k)}_n(j)=0$  for $j\leq 2^{n-1}-k$ and  for $j\geq 2^{n+1}$, we obtain
  the estimate  
\begin{multline*}
\abs{\fh_n(\mu)}
\leq
\abs{1-\mu}^{-M} \sum_{j=2^{n-1}-M}^{2^{n+1}} \abs{(w_n h)^{(M)}(j)}
\\
\leq
C_M\abs{1-\mu}^{-M} \sum_{m=0}^M 2^{-n(M-m)}\sum_{j=2^{n-1}-M}^{2^{n+1}}\abs{h^{(m)}(j)}
\\
\leq
C_M\abs{1-\mu}^{-M} 2^{-nM} 
\sum_{m=0}^M \sum_{j=2^{n-1}-M}^{2^{n+1}}(1+j)^m \abs{h^{(m)}(j)}.
\end{multline*}
From here we get
\begin{multline*}
2^n
\int_{\abs{\theta}\geq2^{-n}} \abs{\fh_n(e^{i\theta})}^q d\theta
\leq
\biggl(
2^{n-nMq}\int_{\abs{\theta}\geq 2^{-n}}
\abs{1-e^{i\theta}}^{-Mq}d\theta
\biggr)
\\
\times
\biggl(
C_M
\sum_{m=0}^M \sum_{j=2^{n-1}-M}^{2^{n+1}}(1+j)^m \abs{h^{(m)}(j)}
\biggr)^q.
\end{multline*}
Since $Mq>1$, the first factor here can be estimated by a constant independent of $n$. 
Combining this with \eqref{ca13a}, we arrive at \eqref{ca22}.
\end{proof}

\begin{proof}[Proof of Theorem~$\ref{thm.a1}$ for $\alpha\geq 1/2$]
Denote 
$$
A_m
=
\sup_{j\geq0}
(j+1)^{1+m}(\log (j+2))^{\alpha}\abs{h^{(m)}(j)}, 
\quad
m=0,\dots,M.
$$
Substituting these bounds into the estimates \eqref{ca21}  and  \eqref{ca22}, we obtain
\begin{align*}
\norm{\fh_n}_{L^\infty}
&\leq
C A_0\jap{n}^{-\alpha}, 
\\
2^n\norm{\fh_n}^q_{L^q}
&\leq
C (A_0+\dots+A_M)^q\jap{n}^{-q \alpha}, 
\end{align*}
if  $2^{n-1}\geq M$. 
Using these estimates and arguing exactly as in the proof of Theorem~\ref{thm.a2}, 
we find that
$$
\babs{h}_p
\leq 
C(A_0+\dots+A_M),
$$
and so by \eqref{ca5b} we are done.
\end{proof}

\begin{proof}[Proof of Theorem~$\ref{cr.a3}$]
Let  $\chi_0\in C^\infty(\bbR_+)$ be as in \eqref{a7b}.
Set $\zeta_N(j)=\chi_0(j/N)$ and consider the truncated sequence $h_{N} = h \zeta_N$.
Then $\Hank(h_N)$ is a finite rank operator. Let us show that $\Hank(h_N )$ converges to $\Hank(h )$ in the quasi-norm 
of $\Sch_{p,\infty}$.
Note that
$$
\sup_{j\geq0}
(j+1)^m
\abs{\zeta^{(m)}_N(j)}\leq C_m, \quad m\geq0,
$$
with constants $C_{m}$ not depending on $N$. Therefore it follows from estimate \eqref{z7} and the Leibniz rule   \eqref{eq:ibp1}  that
$$
\norm{\Hank(h)-\Hank(h_N )}_{\Sch_{p,\infty}}
\leq
C(\alpha)
\sum_{m=0}^M
\sup_{j\geq N/4}
(1+j)^{1+m}(\log( j+2))^{\alpha}\abs{h^{(m)}(j)}, 
$$
where $p=1/\alpha$.
Under the assumptions of Theorem~\ref{cr.a3} the r.h.s. here tends to zero as $N\to \infty$. 
\end{proof}

Let us finally show that the condition  $M(\alpha)=[\alpha] + 1$   in Theorem~\ref{thm.a1} cannot be significantly improved.

\begin{example}\label{exa}
Let $\alpha \geq 2 $. We will construct a sequence $h(j)$ satisfying condition  \eqref{ca5} for  all $m\leq [\alpha] -2$   but such that estimate \eqref{eq:ca5} is violated. For an arbitrary $ \gamma\in ([\alpha]-1,\alpha)$,  consider the lacunary sequence
$$
h(j)=
\begin{cases}
2^{-\gamma n}, & \text{ if $j=2^n$, $n\in\bbN$,}
\\
0 & \text{ otherwise.}
\end{cases}
$$
In this case the iterated differences $h^{(m)}$ do not decay faster than the sequence $h$ itself. 
So for all $m$, we only have
$$
h^{(m)}(j)=O(j^{-\gamma}), \quad j\to\infty.
$$
Thus the hypothesis \eqref{ca5} of Theorem~\ref{thm.a1} is satisfied
for all   $m<\gamma-1$ and hence   for all  $m\leq [\alpha]-2$. 
On the other hand, for our sequence $h$ the polynomial  \eqref{ca8} is 
$\fh_n(z)=2^{-\gamma n}z^{2^n}$. 
So $\abs{\fh_n(e^{i\theta})}=2^{-\gamma n} $  and   the series in the r.h.s. of
\eqref{ca5a} becomes
$$
\sum_{n=0}^\infty 2^n 2^{-\gamma p n}.
$$
This series  diverges for $p=1/\gamma$. 
Therefore according to Remark~\ref{Ne}  (the necessity part of \cite[Theorem 6.2.1]{Peller}), 
we have $\Hank(h)\notin \Sch_{1/\gamma}$. Since $ \Sch_{1/\alpha,\infty}\subset \Sch_{1/\gamma}$ for $\gamma<\alpha$, it follows that $\Hank(h)\notin\Sch_{1/\alpha,\infty}$. Thus, one cannot take $M(\alpha)=[\alpha]-2$ in Theorem~\ref{thm.a1}.
\end{example}


\begin{thebibliography}{6}

\bibitem{BL}
{\sc J.~Bergh, J.~L\"ofstr\"om,}
\emph{Interpolation spaces,}
Springer, 1976. 

\bibitem{BSbook}
{\sc M.~Sh.~Birman, M.~Z.~Solomyak, }
\emph{Spectral theory of selfadjoint operators in Hilbert space,} 
D. Reidel, Dordrecht, 1987. 



\bibitem{Bonsall}
{\sc F.~F.~Bonsall,}
\emph{Some nuclear Hankel operators,}
in: Aspects of Mathematics and its Applications, 
227--238, 
North-Holland Math. Library \textbf{34}, 
North-Holland, Amsterdam, 1986. 

\bibitem{GLP}
{\sc K.~Glover, J.~Lam, J.~R.~Partington,}
\emph{Rational approximation of a class of infinite-dimensional 
systems I: singular values of Hankel operators,}
Math. Control Signals Systems (1990) \textbf{3}, 325--344.

\bibitem {GK} {\sc I. C. Gohberg,  M. G. Kre\u{\i}n}, {\em Introduction to the theory of linear
nonselfadjoint operators in Hilbert space}, Amer. Math. Soc., Providence, Rhode Island, 1970.

\bibitem{Parf}
{\sc O.~G.~Parfenov,}
\emph{Estimates for singular numbers of Hankel operators,}
Math. Notes (1991) \textbf{49}, 610--613.

\bibitem{Peller}
{\sc V.~Peller,}
\emph{Hankel operators and their applications,}
Springer, 2003.



\bibitem{II}
{\sc A.~Pushnitski, D.~Yafaev,}
\emph{Asymptotic behaviour of eigenvalues of Hankel operators,}
in preparation.

\bibitem{Widom}
{\sc H.~Widom,}
\emph{Hankel matrices,}
Trans.   Amer. Math. Soc.  \textbf{121}, no.~1 (1966), 1--35.

\bibitem{Yafaev2}
{\sc D.~R.~Yafaev,}
\emph{Criteria for Hankel operators to be sign-definite,}
preprint, 
arXiv:1303.4040; to appear in A$\&$PDE.  




\end{thebibliography}
\end{document}